\documentclass[11pt,reqno]{amsart}
\usepackage{graphicx}
\usepackage{float}
\usepackage[caption = false]{subfig}
\usepackage{enumerate}
\usepackage{amssymb}
\usepackage{mathtools}
\usepackage{breqn}
\usepackage{array, geometry, graphicx}
\usepackage{amsmath,amsfonts,amssymb,amsthm,mathrsfs,wasysym}
\usepackage{pdfpages}
\newtheorem{theorem}{Theorem}[section]
\newtheorem{corollary}[theorem]{Corollary}
\newtheorem{lemma}[theorem]{Lemma}

\theoremstyle{definition}
\newtheorem{definition}[theorem]{Definition}
\theoremstyle{remark}
\newtheorem{remark}[theorem]{Remark}
\numberwithin{equation}{section}
\DeclareMathOperator{\RE}{Re} \DeclareMathOperator{\IM}{Im}\DeclareMathOperator{\sech}{sech}\DeclareMathOperator{\arctanh}{arctanh}

\begin{document}

\title{On starlike functions associated with a bean shaped domain}
\thanks{The second author is supported by The Council of Scientific and Industrial Research(CSIR). Ref.No.:08/133(0030)/2019-EMR-I}
\author{S. Sivaprasad Kumar}
\address{Department of Applied Mathematics, Delhi Technological University, Delhi--110042, India}
\email{spkumar@dce.ac.in}
\author[Pooja Yadav]{Pooja Yadav}
\address{Department of Applied Mathematics, Delhi Technological University, Delhi--110042, India}
\email{poojayv100@gmail.com}


\begin{abstract}
 In this paper, we introduce and explore a new class of starlike functions denoted by $\mathcal{S}^*_{\mathfrak{B}}$, defined as follows:
 $$\mathcal{S}^*_{\mathfrak{B}}=\{f\in \mathcal{A}:zf'(z)/f(z)\prec \sqrt{1+\tanh{z}}=:\mathfrak{B}(z)\}.$$
 Here, $\mathfrak{B}(z)$ represents a mapping from the unit disk onto a bean-shaped domain. Our study focuses on understanding the characteristic properties of both $\mathfrak{B}(z)$ and the functions in $\mathcal{S}^*_{\mathfrak{B}}$. We derive sharp conditions  under which  $\psi(p)\prec\sqrt{1+\tanh(z)}$ implies $p(z)\prec ((1+A z)/(1+B z))^\gamma$, where $\psi(p)$ is defined as:
 \begin{equation*}
(1-\alpha)p(z)+\alpha p^2(z)+\beta \frac{zp'(z)}{p^k(z)}\quad \text{and}\quad (p(z))^\delta+\beta \frac{zp'(z)}{(p(z))^k}.
\end{equation*}
Additionally, we establish inclusion relations involving $\mathcal{S}^*_{\mathfrak{B}}$ and derive precise estimates for the sharp radii constants of $\mathcal{S}^*_{\mathfrak{B}}$. \end{abstract}
\maketitle

\section{Introduction}

	 Let $\mathbb{D}_r$ denote the open disk in the complex plane with the origin as its center and radius $r$. Specifically, $\mathbb{D}_1=:\mathbb{D}$. Let $\mathcal{A}$ be the class of normalized analytic functions $f(z)$ in $\mathbb{D}$ with the Taylor series expansion $f(z)=z+\sum_{n=2}^{\infty}a_{n}z^n.$ The univalence of $f(z)$ can be determined using the Noshiro and Warschawski criterion, which states that for some $\gamma$ such that $|\gamma|<\pi/2$, if $\RE e^{i \gamma}f'(z)>0$ for all $z$ in the convex domain $D$, then $f(z)$ is univalent in $D$.
The class $\mathcal{S}$ is a subclass of $\mathcal{A}$ consisting of all univalent functions. It plays a vital role in Geometric Function Theory due to its significant applications in areas such as image processing, signal processing, and more. Two well-studied subclasses of $\mathcal{S}$ are the classes of convex and starlike functions, denoted by $\mathcal{C}$ and $\mathcal{S}^*$, respectively. They are characterized analytically as $\RE(zd(\log(zf'(z))))>0$ and $\RE(zd(\log f(z)))>0$, respectively.
In 1992, Ma and Minda \cite{maminda}  unified various subclasses of $\mathcal{S}^*$, by defining the class $\mathcal{S}^*(\psi)$ as follows:

\begin{equation}\label{sp}
	\mathcal{S}^{*}(\psi)=\bigg\{f\in\mathcal{S}:\frac{zf'(z)}{f(z)}\prec\psi(z)\bigg\}.
	\end{equation}
Here, $\psi$ is an analytic, univalent, and starlike function with respect to $\psi(0)=1$. Additionally, $\psi$ satisfy the conditions $\RE\psi(z)>0$, $\psi'(0)>0$, and $\psi(\mathbb{D})$ is symmetric with respect to the real axis. Specializing the function $\psi(z)$ as described in (\ref{sp}) leads to the formation of several important function classes. Among these are $\mathcal{S}^*(\alpha)$, defined by $\mathcal{S}^*((1+(1-2\alpha)z)/(1-z))$, and $\mathcal{SS}^*(\beta)$, represented by $\mathcal{S}^*(((1+z)/(1-z))^\beta)$. Additionally, there are several other subclasses that have been explored by researchers, each having interesting geometrical properties, as illustrated in Table 1.
 \begin{table}[ht]
  \caption{Subclasses of starlike functions obtained for different Choices of $\psi(z)$} 
  \centering 
  \begin{tabular}{lll} 
  \hline 
    {\bf{Class}}  & \textbf{$\psi(z)$} &  \textbf{Reference}      \\ [0.5ex] 
    \hline 
     $\mathcal{S}^*(A,B)$     &   $(1+Az)/(1+Bz)$, $-1\leq B<A\leq1$ &  \cite{janowski} W. Janowski.  \\
      $\mathcal{S}^*_{\varrho}$   &   $1+z e^z$   &   \cite{gangania1}  S. S. Kumar and G. Kamaljeet   \\
     $\mathcal{S}^*_{SG}$    &   $2/(1+e^{-z})$   &  \cite{pri}  P. Goel and S. Kumar   \\
     $\mathcal{S}^*_{e}$       &  $e^{z}$  &   \cite{exp}  R. Mendiratta et al. \\
     $\mathcal{S}^*_{\leftmoon}$                &  $z+\sqrt{1+z^2}$  &  \cite{raina}  R. K. Raina and J Sok\'{o}\l \\
     $\mathcal{S}^*_{L}$     &   $\sqrt{1+z}$  &   \cite{sokol1}  J. Sok\'{o}\l\; and Stankiewicz  \\


    \hline 
  \end{tabular}
\end{table}

Moreover, these classes have been thoroughly investigated in prior studies such as \cite{kanaga, exact, oblique, poonam}. Inspired by the findings presented in these articles, we introduce a new class associated with the bean-shaped domain represented by the function $\mathfrak{B}(z):=\sqrt{1+\tanh{z}}$.
\begin{definition}
Let $\mathcal{S}^*_{\mathfrak{B}}$ be the class of normalized starlike functions, defined as
follows:
	\begin{equation*}
\mathcal{S}^*_{\mathfrak{B}}:=\left\{f\in\mathcal{S}:\dfrac{zf'(z)}{f(z)}\prec\mathfrak{B}(z):=\sqrt{1+\tanh{z}}\right\}.
	\end{equation*} 
 \end{definition}
Note that $\mathfrak{B}(z):=\sqrt{1+\tanh{z}}=\sqrt{2/(1+e^{-2z})}$ and $\mathfrak{B}(z)$ conformally maps $\mathbb{D}$ onto the region $$\Omega_{\mathfrak{B}}:=\left\{w\in\mathbb{C}:\left|\log\left(\dfrac{w^2}{2-w^2}\right)\right|<2\right\}.$$  The functions in the classes $\mathcal{S}^{*}_{\mathfrak{B}}$ can be represented through an integral formula as follows:\\
	 A function $f\in\mathcal{S}^{*}_{\mathfrak{B}}$ if and only if there exists an analytic function $p(z)\prec\mathfrak{B}(z)$ such that
	\begin{equation}\label{4}
	f(z)=z \exp\int_{0}^{z}\frac{p(t)-1}{t}dt.
	\end{equation}
	By taking $p(z) =\mathfrak{B}(z)$ in \eqref{4}, we obtain an extremal function, defined as:
	\begin{equation}\label{5}
	f_{0}(z)=z+\frac{z^2}{2}+\frac{z^{3}}{16}-\frac{13z^{4}}{288}-\frac{11z^{5}}{1152}+\cdots\in \mathcal{S}^{*}_{\mathfrak{B}}.
	\end{equation} 
The class $\mathcal{S}^{*}_{\mathfrak{B}}$ and its integral representation \eqref{4} offer valuable insights into the properties and behavior of starlike functions associated with the intriguing bean-shaped domain. The extremal function $f_{0}(z)$ represents a notable example within this class, paving the way for further exploration and analysis.
The present investigation explores the properties of $\mathfrak{B}(z)$ and the class $\mathcal{S}^*_{\mathfrak{B}}$. It establishes inclusion relations between $\mathcal{S}^*_{\mathfrak{B}}$ and well-known classes, accompanied by diagrammatic representations. The study also derives sharp conditions  under which  $\psi(p)\prec\sqrt{1+\tanh(z)}$ implies $p(z)\prec ((1+A z)/(1+B z))^\gamma$, where $\psi(p)$ is defined as:
 \begin{equation*}
(1-\alpha)p(z)+\alpha p^2(z)+\beta \frac{zp'(z)}{p^k(z)}\quad \text{and}\quad (p(z))^\delta+\beta \frac{zp'(z)}{(p(z))^k}.
\end{equation*} Additionally, find sharp $\mathcal{S}^*_{\mathfrak{B}}$-radii estimates for certain geometrically
defined function classes available in the literature.

\section{Geometric Properties}
The function $\mathfrak{B}(z)=\sqrt{1+\tanh{z}}$ is analytic, and its derivative satisfies $\operatorname{Re}(\mathfrak{B}'(z))>\mathfrak{B}'(0)=0.158$, along with $\operatorname{Re}\left(\frac{z\mathfrak{B}'(z)}{\mathfrak{B}(z)-\mathfrak{B}(0)}\right)>0.483$. Consequently, $\mathfrak{B}(z)$ is both univalent and starlike with respect to the point $\mathfrak{B}(0)=1$.

For $|z|= r$, the expression of $\mathfrak{B}(re^{i\theta})$ is given by:
\begin{align}\label{beanbary}\mathfrak{B}(re^{i\theta})&=\sqrt{\frac{2}{1+e^{-2r(\cos{\theta}+i\sin{\theta})}}}\nonumber\\
	&=\dfrac{\sqrt{2}\left(\cos(T_{r,\theta})+i\sin(T_{r,\theta})\right)}{\left(\left(1+e^{-2r\cos{\theta}}\cos(2r\sin{\theta})\right)^2+e^{-4r\cos{\theta}}\sin^{2}(2r\sin{\theta})\right)^{\frac{1}{4}}}\nonumber\\
	&=\frac{e^{r\cos{\theta}}\left(M_{r,\theta}+\sqrt{M_{r,\theta}^2+N_{r,\theta}^2}+i\; N_{r,\theta}\right)}{\sqrt{M_{r,\theta}^2+N_{r,\theta}^2}\sqrt{M_{r,\theta}+\sqrt{M_{r,\theta}^2+N_{r,\theta}^2}}},
	\end{align}
where $\theta\in[0,2\pi)$ and
\begin{gather*}
M_{r,\theta}=e^{2r\cos{\theta}}+\cos(2r\sin{\theta}),\quad N_{r,\theta}=\sin(2r\sin{\theta}) \
\intertext{and}
T_{r,\theta}=\frac{1}{2}\arctan\left(\frac{N_{r,\theta}}{M_{r,\theta}}\right).
\end{gather*}
In particular, let $M_\theta:=M_{1,\theta}$ and $N_\theta:=N_{1,\theta}$. The function $\mathfrak{B}$ is symmetric about the real axis as $\operatorname{Re}(\mathfrak{B}(r,\theta))=\operatorname{Re}(\mathfrak{B}(r,-\theta))$. Consequently, $\mathcal{S}^*_{\mathfrak{B}}$ is a Ma-Minda starlike class.

We begin by establishing the following result, which addresses the radius of convexity of the function $\mathfrak{B}(z)$:

\begin{theorem}
$\mathfrak{B}(z)$ is convex for $|z|\leq r^*$, where $r^*\simeq0.7074$ is the smallest positive root of $1+2e^{2r}+e^{4r}-(-1+e^{2r}+2e^{4r})r=0$.
\end{theorem}

\begin{proof}
Considering $\mathfrak{B}(z)=\sqrt{\frac{2}{1+e^{-2z}}}$ and $z=re^{i t}$, We obtain 
\begin{align*}
\RE\left(1+\dfrac{z\mathfrak{B}''(z)}{\mathfrak{B}'(z)}\right) &= \RE\left(\dfrac{1+z+e^{2z}(1-2z)}{1+e^{2z}}\right)\\
&=\dfrac{G_n(r,t)}{G_d(r,t)}=:G(r,t),
\end{align*}
where
{\small
\begin{equation*}
G_n(r,t):=1+r\cos{t}+e^{4r\cos{t}}(1-2r\cos{t})+e^{2r\cos{t}}(2\cos(2r\sin{t})+r(\cos(t-2r\sin{t})-2\cos(t+2r\sin{t})))
\end{equation*}
}and
\begin{equation*}
G_d(r,t):=(1+ e^{2\cos{t}}\cos^2(2\sin{t})+e^{4\cos{t}}\sin^2(2\sin{t})).
\end{equation*}
Since $G(r,t)=G(r,-t)$,  $G$ is symmetric about the real axis, we can focus on $G$ for $t\in[0,\pi]$. To find the radius of convexity of $\mathfrak{B}(z)$, we need to find the smallest $r^*\in(0,1)$ such that $G(r,t)>0$ for $r\in(0,r^*)$. Since $G_d(r,t)>0$ for $r\in(0,1)$ and $t\in[0,\pi]$, it suffices to show that $G_n(r,t)>0$ for $r\in(0,r^*)$.
Notably, for any fixed $r$, $G_n(r,t)$ attains its minimum at $t=0$, resulting in the minimum value $G_n(r,0)=1+2e^{2r}+e^{4r}-(-1+e^{2r}+2e^{4r})r$. Since $G_n(0,0)>0$, and $r^*$ is the smallest positive root of $G_n(r,0)=0$, this directly establishes the desired result.
\end{proof}

Using elementary calculus and (\ref{beanbary}), we can derive our next lemma, which provides sharp bounds associated with $\mathfrak{B}(z)$ for $|z|=1$.

\begin{lemma}{\label{bds}} For $|z|=1$, the following sharp bounds hold for $\mathfrak{B}(z)$:
\begin{enumerate}
\item[(i)]
If we define $R(\theta)=\frac{e^{\cos{\theta}}\sqrt{M_\theta+\sqrt{M_\theta^2+N_\theta^2}}}{\sqrt{M_\theta^2+N_\theta^2}}$, then $\sqrt{\frac{2}{1+e^2}}\leq\operatorname{Re}(\mathfrak{B}(z))\leq R(\theta_{0})$, where $R(\theta_{0})\simeq1.38846$ at $\theta_{0}\simeq1.15197$.
\item[(ii)] If we define $I(\theta)=\frac{e^{\cos{\theta}}N_\theta}{\sqrt{M_\theta^2+N_\theta^2}\sqrt{M_\theta+\sqrt{M_\theta^2+N_\theta^2}}}$, then $|\operatorname{Im}(\mathfrak{B}(z))|\leq I(\theta_{1})$, where $I(\theta_{1})\simeq0.69949$ at $\theta_{1}\simeq1.72466$.
\item[(iii)] If we define $A(\theta)=\frac{2e^{2\cos{\theta}}}{\sqrt{M_\theta^2+N_\theta^2}}$, then $\sqrt{\frac{2}{1+e^2}}\leq|\mathfrak{B}(z)|\leq \sqrt{A(\theta_{2})}$, where $A(\theta_{2})\simeq2.0694$ at $\theta_{2}\simeq1.31364.$
\item[(iv)] $|\operatorname{arg}(\mathfrak{B}(z))|\leq \dfrac{\beta \pi}{2},$ where  $\beta=0.43849139.$
\end{enumerate}
\end{lemma}

This lemma provides a comprehensive set of bounds for the real and imaginary parts, modulus, and argument of $\mathfrak{B}(z)$ when $|z|=1$. These sharp bounds are expressed in terms of trigonometric and exponential functions and offer valuable insights into the behavior of $\mathfrak{B}(z)$ on the unit circle.
 \begin{figure}[h]
\begin{tabular}{c}
\includegraphics[scale=0.3]{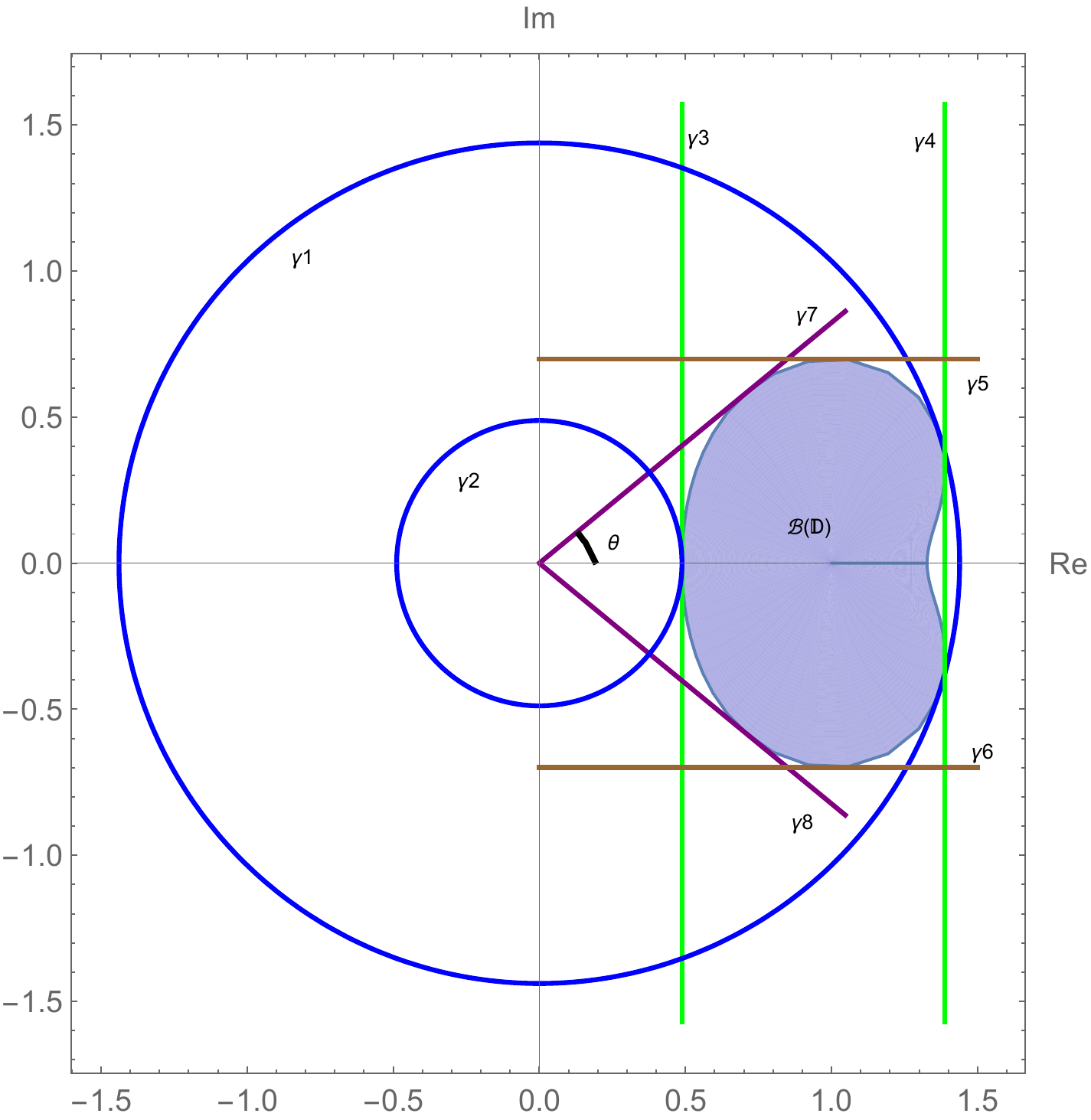}
\end{tabular}
\begin{tabular}{l}
\parbox{0.25\linewidth}{
{\bf \underline{Legend} - }\\
$\mathfrak{B}(z)=\sqrt{1+\tanh{z}}$}\\\vspace{0.15cm}
$\gamma_{1}:= \{w\in\mathbb{C}:|w|=2.0694\}$\\ \vspace{0.15cm}  
$\gamma_{2}:=  \{w\in\mathbb{C}:|w|=\sqrt{2/(1+e^2)}\}$\\ \vspace{0.15cm}
$\gamma_{3}:=\{w\in\mathbb{C}:\RE{w}=\sqrt{2/(1+e^2)}\}$ \\ \vspace{0.15cm}  
$\gamma_{4}:= \{w\in\mathbb{C}:\RE{w}=1.38846\} $\\ \vspace{0.15cm}
$\gamma_{5}:= \{w\in\mathbb{C}:\IM{w}=0.69949,\;\RE{w}>0\}$ \\ \vspace{0.15cm}
$\gamma_{6}:= \{w\in\mathbb{C}:\IM{w}=-0.69949,\;\RE{w}>0\}$\\ \vspace{0.15cm}
$\arg{\gamma_{7}}=-\arg{\gamma_8}=\theta,$\\ \vspace{0.15cm} $\theta\simeq39.4642^\circ\simeq0.438491\pi/2$ 
\end{tabular}

\caption{Graph indicating the sharpness of various bounds  (in context of Lemma \ref{bds}) associated with  $\mathfrak{B}(\mathbb{D}).$}
\label{f2}
\end{figure}

 \begin{remark}\label{minrem}
It is evident that within each circle $|z| = r < 1$, the minimum value of both $\RE \mathfrak{B}(z)$ and $|\mathfrak{B}(z)|$ occurs precisely at the point $z = -r$.
\end{remark}
 Now using results in \cite{maminda}, we have the following basic results for functions in $\mathcal{S}^*_{\mathfrak{B}}$:
 \begin{theorem}
    Let $f\in\mathcal{S}^*_{\mathfrak{B}}$ and $f_0$ be as defined in (\ref{5}). For any $|z|=r<1$, the following fundamental results hold:
      \begin{enumerate}
      \item[(i)]  $\frac{f(z)}{z}\prec \frac{f_0(z)}{z}$ and $\frac{zf'(z)}{f(z)}\prec \frac{zf_0'(z)}{f_0(z)}$ (Subordination Theorem).
          \item[(ii)] $-f_0(-r)\leq|f(z)|\leq f_0(r)$ (Growth Theorem).
          \item[(iii)] $|\arg\left(\frac{f(z)}{z}\right)|\leq\max_{|z|=r}\arg\left(\frac{f_0(z)}{z}\right)$ (Rotation Theorem). 
          \item[(iv)]  Either $f$ is a rotation of $f_0$ or $f(\mathbb{D})\supset\{w\in\mathbb{C}:|w|\leq -f_0(-1)\} $ (Covering Theorem).
      \end{enumerate}
     Equality holds for some non-zero $z$ in $(ii)$ and $(iii)$ if and only if $f$ is a rotation of $f_0.$
  \end{theorem}

 In their work \cite{maminda}, Ma-Minda established the distortion theorem for the class $\mathcal{S}^*(\psi)$, subject to the condition:
\begin{equation}\label{discon}
\psi(-r) = \min_{|z|=r}|\psi(z)| \leq |\psi(z)| \leq \max_{|z|=r}|\psi(z)| = \psi(r).
\end{equation}
However, it is worth noting that the function $\mathfrak{B}$ does not satisfy condition (\ref{discon}) since $\max_{|z|=r}|\mathfrak{B}(z)| \neq \mathfrak{B}(r)$.
Recently, Gangania and Kumar \cite{gangania} modified the distortion theorem by relaxing the condition (\ref{discon}), leading to a more general result, as follows:
  \begin{theorem}\cite{gangania}
Let $\psi$ be a Ma-Minda function. Suppose $\min_{|z|=r}|\psi(z)| = |\psi(z_1)|$ and $\max_{|z|=r}|\psi(z)| = |\psi(z_2)|$, where $z_1 = re^{i\theta_1}$ and $z_2 = re^{i\theta_2}$ for some $\theta_1, \theta_2 \in [0,\pi]$. For $f\in\mathcal{S}^*(\psi)$ and $|z_0|=r<1$, the following inequality holds:
  \begin{equation*}
    |\psi(z_1)| \left(\dfrac{-f_0(-r)}{r}\right) \leq |f'(z_0)| \leq \left(\dfrac{f_0(r)}{r}\right) |\psi(z_2)|.
\end{equation*}
  \end{theorem}
  To establish the distortion theorem for functions in $\mathcal{S}^*_{\mathfrak{B}}$, we consider the expression for $|\mathfrak{B}(re^{i\theta})|^2$, denoted as $g_r(\theta)$:  
  \begin{equation*}
      |\mathfrak{B}(re^{i\theta})|^2=\dfrac{2e^{2r\cos\theta}}{\sqrt{1+e^{4r\cos\theta}+2e^{2r\cos\theta}\cos(2r\sin\theta)}}=:g_r(\theta).
  \end{equation*}
  Next, we calculate the derivative of $g_r(\theta)$,
  $$g_r'(\theta)=-\dfrac{4e^{2r\cos\theta}r(\sin\theta+e^{2r\cos\theta}\cos(2r\sin\theta)\sin\theta-e^{2r\cos\theta}\sin(2r\sin\theta)\cos\theta)}{(1+e^{4r\cos\theta}+2e^{2r\cos\theta}\cos(2r\sin\theta))^{3/2}}.$$ 
 By simple observation, we can see that, the critical points of $g_r(\theta)$ are the zeroes of
 $$h_r(\theta):=\sin\theta+e^{2r\cos\theta}\cos(2r\sin\theta)\sin\theta-e^{2r\cos\theta}\sin(2r\sin\theta)\cos\theta.$$ 
Clearly, $\theta=0$ and $\pi$ are zeroes of $h_r(\theta)$ for all $r$. Remark \ref{minrem} confirms that $|\mathfrak{B}(z)|$ attains its minimum at $\theta=\pi$ for $|z|=r$. However, the maximum is not necessarily attained at $\theta=0$, as $h_r(\theta)$ possesses roots other than $0$ and $\pi$ for $r> r_0\simeq0.639$.
Consequently, we can now state the distortion theorem for functions in $\mathcal{S}^*_{\mathfrak{B}}$ as follows:
\begin{theorem}(\textit{Distortion Theorem})
Let $\max_{|z|=r}|\mathfrak{B}(z)|=|\mathfrak{B}(z_0)|$, where $z_0=re^{i \theta_0}$ for some $\theta_0\in[0,\pi]$. For $f\in\mathcal{S}^*_{\mathfrak{B}}$ and $f_0$ as defined in (\ref{5}), the following inequalities hold for $|z|=r<1$:
\begin{align*}
   f_0'(-r) & \leq |f'(z)| \leq \frac{|\mathfrak{B}(r e^{i \theta_0})| f_0(r)}{r} \quad \text{for} \quad r>r_0\simeq0.639, \\
f_0'(-r) & \leq |f'(z)| \leq f_0'(r) \quad \text{for} \quad r\leq r_0.
\end{align*}
\end{theorem}
 
 The following lemma aims to find the maximal disk centered at the
sliding point $(a, 0)$ on the real line, that can fit within $\mathfrak{B}(\mathbb{D}).$
  
\begin{lemma}\label{circle}
	Let $\sqrt{2/(1+e^2)}<\alpha<e\sqrt{2/(1+e^2)}$.
	Then we have the inclusion:
\begin{equation}\label{bdisk}
    \{w:|w-\alpha|<r_{\alpha}\}\subset \mathfrak{B}(\mathbb{D}),
\end{equation}
where $r_{\alpha}$ is defined as follows:
\begin{equation}\label{ralpha}
    r_{\alpha}=\begin{cases}
        \alpha-\sqrt{\frac{2}{(1+e^2)}},&\sqrt{\frac{2}{(1+e^2)}}<\alpha\leq\frac{1+e}{\sqrt{2(1+e^2)}}\\
        e\sqrt{\frac{2}{(1+e^2)}}-\alpha,& \frac{1+e}{\sqrt{2(1+e^2)}}\leq\alpha<e\sqrt{\frac{2}{(1+e^2)}}.
    \end{cases}
\end{equation}

\end{lemma}

\begin{proof}
Since (\ref{beanbary}) is the boundary curve of the domain $\mathfrak{B}(\mathbb{D}_r)$, which is symmetric about the real axis, we can focus on $0\leq\theta\leq\pi$. We begin by evaluating the square of the distance from $(\alpha,0)$ to the points on the curve $\mathfrak{B}(e^{i\theta})$, denoted as $d(\theta)$:
\begin{equation}\label{dis}
d(\theta) = \frac{e^{2\cos{\theta}}N^2_\theta}{(M^2_\theta+N^2_\theta)(M_\theta+\sqrt{M^2_\theta+N^2_\theta})} + \left(\alpha-\frac{e^{\cos{\theta}}\sqrt{M_\theta+\sqrt{M^2_\theta+N^2_\theta}}}{\sqrt{M^2_\theta+N^2_\theta}}\right)^2.
\end{equation}	We then find the critical points by solving $d'(\theta)=0$, which yields $\theta=0$, $\theta_\alpha$, and $\pi$. Furthermore, we observe that $d(\theta)$ is increasing in $[0,\theta_{\alpha}]$ and decreasing in $[\theta_{\alpha},\pi]$, where $\theta_\alpha$ depends on the parameter $\alpha$.\\
\textbf{Case (i):} If $\sqrt{2/(1+e^2)}<\alpha<(1+e)/\sqrt{2(1+e^2)}$, we have $d(0)-d(\pi) > 0$, indicating that the minimum of $d(\theta)$ is attained at $\pi$. Thus, in this case:  
$$r_{\alpha}=\sqrt{d(\pi)}=\alpha-\sqrt{\frac{2}{1+e^2}}.$$	
\textbf{Case (ii):} If $(1+e)/\sqrt{2(1+e^2)}<\alpha<e\sqrt{2/(1+e^2)}$, we find $d(0)-d(\pi) < 0$, implying that the minimum of $d(\theta)$ is attained at $\theta=0$. Consequently, we obtain:	$$r_{\alpha}=\sqrt{d(0)}=e\sqrt{\frac{2}{(1+e^2)}}-\alpha.$$
That completes the proof.
\end{proof}

\begin{corollary}If $\sqrt{2/(1+e^2)}<\alpha<e\sqrt{2/(1+e^2)}$, then
    $$ \mathfrak{B}(\mathbb{D})\subset\{w\in\mathbb{C}:|w-\alpha|\leq\sqrt{d(\theta_{\alpha})}\},$$where $d(\theta)$ is given by (\ref{dis}) and $\theta_{\alpha}$ is the zero of $d'(\theta)$ other than $0$ and $\pi$.
\end{corollary}
\begin{remark}\label{rem}
$(i)$ Let $\alpha_0:=(1+e)/\sqrt{2(1+e^2)}$. The largest disk inside $\mathfrak{B}(\mathbb{D})$ is $\{w\in\mathbb{C}:|w-\alpha_0|<r_{\alpha_0}\}$.\\
$(ii)$ The smallest disk containing $\mathfrak{B}(\mathbb{D})$ is $\{w\in\mathbb{C}:|w-1.006|<\sqrt{d(\theta_{1})}=0.69949\}$, where $\theta_1\simeq1.72466$, as given in Lemma \ref{bds}.\\
$(iii)$ The smallest disk containing $\mathfrak{B}(\mathbb{D})$ taking $1$ as center is $\{w\in\mathbb{C}:|w-1|<0.699517\}.$
\end{remark}
In view of Lemma \ref{circle}, we establish the following result:

\begin{theorem}\label{thmjan}
	Let  $p(z):=(1+Az)/(1+Bz)$, where $-1<B<A\leq1$, then   $p(z)\prec\mathfrak{B}(z)$ if and only if
	\begin{equation}\label{jano}
	A\leq\begin{cases}1-\sqrt{\frac{2}{1+e^2}}(1-B)&\quad\text{if}\quad \frac{1-AB}{1-B^2}\leq\frac{1+e}{\sqrt{2(1+e^2)}}\\
	e\sqrt{\frac{2}{1+e^2}}(1+B)-1&\quad \text{if}\quad\frac{1-AB}{1-B^2}\geq\frac{1+e}{\sqrt{2(1+e^2)}}.
	\end{cases}
		\end{equation}
\end{theorem}
\begin{proof}
For $-1<B<A\leq1$, the function $p(z)=(1+Az)/(1+Bz)$ maps the unit disc $\mathbb{D}$ onto the disc $D(c,r)$, where
\begin{equation*}
c=\frac{1-AB}{1-B^2},\quad r=\frac{A-B}{1-B^2}.
\end{equation*}
Since $1\in D(c,r)$, we have $p(z)\prec \mathfrak{B}(z)$ if and only if $D(c,r)\subset\mathfrak{B}(\mathbb{D})$.
Using Lemma \ref{circle}, we find that $p(z)\prec \mathfrak{B}(z)$ if and only if $D(c,r)\subset D((1+e)/\sqrt{2(1+e^2)},\sqrt{2/(1+e^2)}(e-1)/2)$. This is equivalent to the inequality:
\begin{equation}
\bigg|c-\frac{1+e}{\sqrt{2(1+e^2)}}\bigg|\leq\sqrt{\frac{2}{1+e^2}}\bigg(\frac{e-1}{2}\bigg)-r.
\end{equation}
Now, considering the conditions given in (\ref{ralpha}), we establish the result (\ref{jano}).
\end{proof}

 From Theorem \ref{thmjan}, we deduce the following corollary, which ensures that $\mathcal{S}^*_{\mathfrak{B}}$ is nonempty.
\begin{corollary}
	 If the conditions on $A,B$ given in Theorem \ref{thmjan} hold, then $\mathcal{S}^*[A,B]\subset\mathcal{S}^*_{\mathfrak{B}}.$
\end{corollary}

\begin{theorem}
    Let $f\in\mathcal{S}^*_{\mathfrak{B}}$, then $f\in\mathcal{S}^*_e$, which is $\mathcal{S}^*_e:=\{f\in\mathcal{A};zf'(z)/f(z)\prec e^z\}$.
\end{theorem}
 
\begin{proof}
 Let $f\in\mathcal{S}^*_{\mathfrak{B}}$, then the smallest disk containing $\mathfrak{B}(\mathbb{D})$ taking $1$ as center is $\{w\in\mathbb{C}:|w-1|<0.699517\}$, given in Remark \ref{rem}. According to   \cite[Lemma 2.2]{exp}, we have that the disk $\{w\in\mathbb{C}:|w-1|<1.71828\}\subset \Lambda$, where  $\Lambda:=\{w\in\mathbb{C}:|\log w|<1\}$ is the image domain of $e^z$ on $\mathbb{D}.$ Therefore the result holds.
\end{proof}
Now, we explore the inclusion relations of $\mathcal{S}^*_{\mathfrak{B}}$ with respect to several other classes, including starlike of reciprocal order $\alpha$, Uralegaddi class of order $\beta$, $k-$starlike class, and the class of functions having a parabolic region. Let's define these classes:
\begin{gather*}
\mathcal{RS}^{*}(\alpha):=\left\{f\in\mathcal{S}:\RE \frac{f(z)}{zf'(z)}>\alpha,\quad 0\leq\alpha<1 \right\},\\
\mathcal{M}(\beta):=\left\{f\in\mathcal{S}:\RE \frac{zf'(z)}{f(z)}<\beta,\quad \beta>1\right\},\\
k-\mathcal{ST}:=\left\{f\in\mathcal{S}:\RE \frac{zf'(z)}{f(z)}>k\bigg|\frac{zf'(z)}{f(z)}-1\bigg|,\quad k\geq0\right\}
\intertext{and}
    \mathcal{SP}(\rho):=\left\{f\in\mathcal{S}:\RE\frac{zf'(z)}{f(z)}+\rho>\bigg|\frac{zf'(z)}{f(z)}-\rho\bigg|,\quad \rho>0\right\}.
    \end{gather*}

 \begin{figure}[ht]
   \begin{tabular}{c}
        \includegraphics[height=10cm,width=6cm]{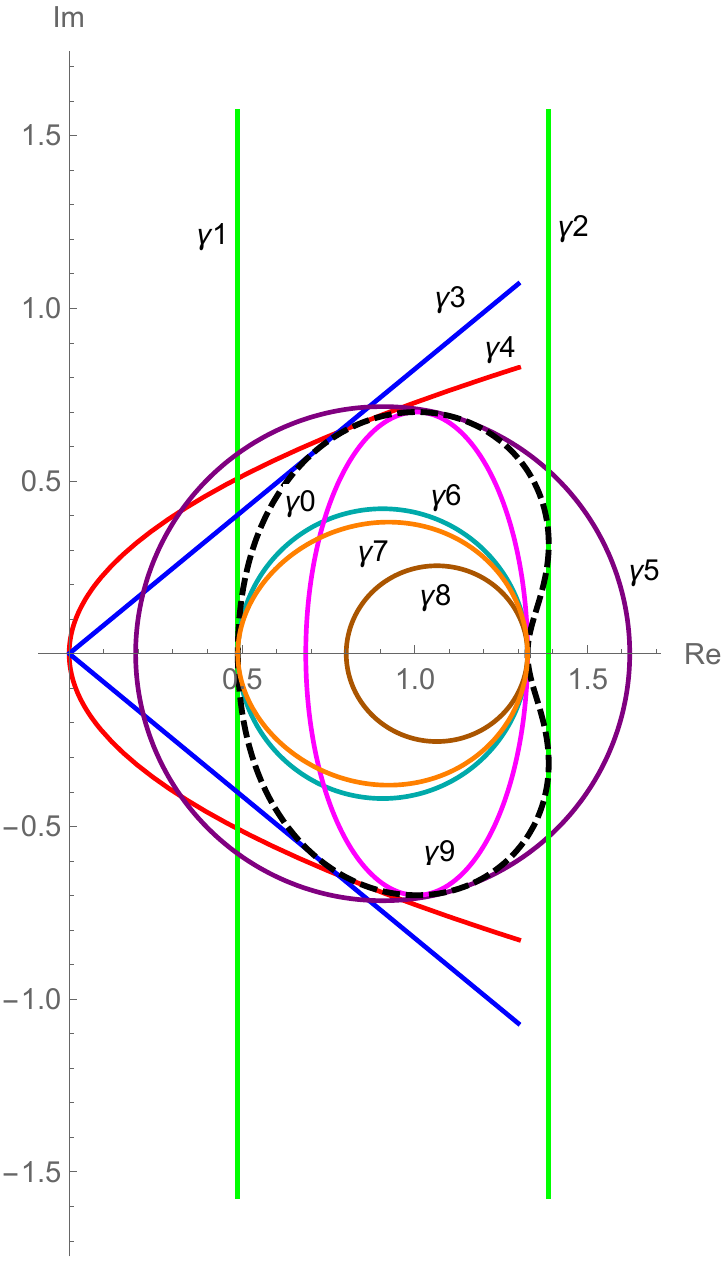}
        \end{tabular}
        \begin{tabular}{l}
\parbox{0.25\linewidth}{
{\bf \underline{Legend} - }\\
$\gamma_{0}:\mathfrak{B}(z)=\sqrt{1+\tanh{z}}$}\\\vspace{0.15cm}
$\gamma_{1}: \RE{w}=\sqrt{2/(1+e^2)}$\\ \vspace{0.15cm}  
$\gamma_{2}:  \RE{w}=e\sqrt{2/(1+e^2)}$\\ \vspace{0.15cm}
$\gamma_{3}:|\arg{w}|=0.438491\pi/2$ \\ \vspace{0.15cm}  
$\gamma_{4}:|w-0.13186|=\RE{w}+0.13186 $\\ \vspace{0.15cm}
$\gamma_{5}: |w-1.006|=0.69949$ \\ \vspace{0.15cm}
$\gamma_{6}: |w-\frac{1+e}{\sqrt{2(1+e^2)}}|=0.41949$\\ \vspace{0.15cm}
$\gamma_{7}: \sqrt{1+\frac{e^2-1}{e^2+1}z}$ \\ \vspace{0.15cm}
$\gamma_{8}: \RE w>\frac{2e}{2e-\sqrt{2(1+e^2)}}|w-1|$\\ \vspace{0.15cm}
$\gamma_{9}:\frac{(\RE{w}-1.006)^2}{(0.321)^2}+\frac{(\IM{w})^2}{(0.69949)^2}$
\end{tabular}
        \caption{Boundary curves of best dominants and subordinants of $\mathfrak{B}(z)$.}
        \label{bean}
    \end{figure}
\begin{theorem}The class $\mathcal{S}^*_{\mathfrak{B}}$ satisfies the following inclusion relations: \begin{itemize}
\item[$(i)$] $\mathcal{S}^*_{\mathfrak{B}}\subset\mathcal{S^{*}}(\alpha)\subset\mathcal{S}^*$, whenever $\alpha=\sqrt{2/(1+e^2)}$.
		\item [$(ii)$]
		$\mathcal{S}^*_{\mathfrak{B}}\subset\mathcal{RS^{*}}(1/\beta)\subset \mathcal{M}(\beta)$, whenever $\beta\geq R(\theta_{0})$ (cf. Lemma \ref{bds}).
		\item [$(iii)$]
		$\mathcal{S}^*_{\mathfrak{B}}\subset\mathcal{SS}^*(\beta)\subset\mathcal{S}^*$ for $\beta_{0}\leq\beta<1$, where $\beta_{0}=0.438491$ (cf. Lemma \ref{bds}).
		\item[$(iv)$] $\mathcal{S}^*_{\mathfrak{B}}\subset\mathcal{SP}(\rho)$, whenever $\rho\geq0.13186.$
		\item[$(v)$] $k-\mathcal{ST}\subset\mathcal{S}^*_{\mathfrak{B}}$, whenever $k\geq 2e/(2e-\sqrt{2(1+e^2)})$.
		\end{itemize}
\end{theorem}

\begin{proof}
	Let $f\in \mathcal{S}^*_{\mathfrak{B}}$. From Lemma \ref{bds}, we have, $$\sqrt{\dfrac{2}{1+e^2}}<\RE\dfrac{zf'(z)}{f(z)}<R(\theta_0)$$ and $$\left|\arg\left(\dfrac{zf'(z)}{f(z)}\right)\right|\leq\frac{0.438491\pi}{2}$$ which eventually yields $(i)$, $ii)$ and $(iii)$. \\
	$(iv)$ Let $f\in\mathcal{SP}(\rho)$ and
	\begin{equation*}
	\Gamma_{\rho}	:= \{w:|w-\rho|<\RE w+\rho\}=\{w=u+iv:v^2<4\rho u\}.
	\end{equation*}
	Now $\mathfrak{B}(\mathbb{D})\subset\Gamma_{\rho}$, whenever $\rho>v^2/4u$, where $u:=\RE\mathfrak{B}(z)$ and $v:=\IM\mathfrak{B}(z)$. For $z=e^{i\theta}$, using (\ref{beanbary}) we have
	\begin{equation*}
	T(\theta):=\frac{e^{\cos{t}}N^2_\theta}{4\sqrt{M^2_\theta +
			N^2_\theta} \left(M_\theta +
		\sqrt{M^2_\theta + N^2_\theta}\right)^{3/2}}<\rho
	\end{equation*}
	Further $T'(\theta)=0$ if and only if $\theta\in\{0,\theta_{0},\pi\}$, where $\theta_{0}\simeq1.8603$ and maximum of $T(\theta)$ is attained at $\theta_{0}$. Therefore, $\rho>T(\theta_{0})\simeq0.13186.$ Now since $\mathcal{SP}(\rho_{1})\subset\mathcal{SP}(\rho_{2})$ whenever $\rho_{1}<\rho_{2}$. Therefore, $\mathcal{S}^*_{\mathfrak{B}}\subset\mathcal{SP}(\rho)$, for $\rho\geq0.13186.$
	\\
	$(v)$  Let $f\in k-\mathcal{ST}$ and $\Lambda_{k}:=\{w\in\mathbb{C}:\RE w>k|w-1|\}$. For $k>1$, the boundary $\partial\Lambda_{k}$ is an ellipse
	\begin{equation*}
	\frac{(x-x_{0})^2}{u^2}+\frac{(y-y_{0})^2}{v^2}=1,
	\end{equation*}
	where
	\begin{equation*}
	x_{0}=\frac{k^2}{k^2-1},\quad y_{0}=0,\quad u=\frac{k}{k^2-1}\quad \text{and}\quad v=\frac{1}{\sqrt{k^2-1}}.
	\end{equation*}
	As $u>v$, so this is an horizontal ellipse.  For the ellipse $\Lambda_{k}$ to lie inside $\mathfrak{B}(\mathbb{D})$, we must ensure that
	\begin{equation*}
	\sqrt{\frac{2}{1+e^2}}\leq x_{0}-u \quad \text{and}\quad x_{0}+u\leq e\sqrt{\frac{2}{1+e^2}},
	\end{equation*}
	which holds for \begin{equation*}
	k\geq \max\bigg\{\frac{2}{2-\sqrt{2(1+e^2)}},\frac{2e}{2e-\sqrt{2(1+e^2)}}\bigg\}=\frac{2e}{2e-\sqrt{2(1+e^2)}}.
	\end{equation*}
	Clearly $\Lambda_{k_{1}}\subset\Lambda_{k_{2}}$ for $k_{1}<k_{2}$. Therefore, $k-\mathcal{ST}\subset\mathcal{S}^*_{\mathfrak{B}}$ for $k\geq2e/(2e-\sqrt{2(1+e^2)})$.
\end{proof}

Let us recall that, in \cite{lemni}, Sok\'{o}\l \;considered the starlike class $\mathcal{S}^*(q_c)=\{f\in\mathcal{A};zf'(z)/f(z)\prec\sqrt{1+cz},\;c\in(0,1]\}$ associated with the right loop of the Cassinian Ovals when $c\in(0,1)$ and right half of the Lemniscate of Bernoulli when $c=1.$ Here we obtain the following inclusion relation between the classes $\mathcal{S}^*(q_c)$ and $\mathcal{S}^*_{\mathfrak{B}}$.

\begin{theorem}
    Let $f\in\mathcal{S}^*(q_c)$, $c\in(0,1]$, then $f\in\mathcal{S}^*_{\mathfrak{B}}$ if and only if $c<(e^2-1)/(e^2+1)$. 
\end{theorem}

\begin{proof}
    Let $f\in\mathcal{S}^*(q_c)$ and $\Lambda_c=\{w\in\mathbb{C}:\RE w>0,|w^2-1|<c\}$ be the region bounded by the cassinian ovals. Then $$\sqrt{1-c}<\RE\dfrac{zf'(z)}{f(z)}<\sqrt{1+c}.$$ Now to find the range of $c$ for which $\Lambda_c$ lies inside $\mathfrak{B}(\mathbb{D})$, we must ensure that 
    \begin{equation*}
        \sqrt{\dfrac{2}{1+e^2}}<\sqrt{1-c}\;\;\;\text{and}\;\;\;e\sqrt{\dfrac{2}{1+e^2}}>\sqrt{1+c}.
    \end{equation*}Clearly, if $c\in(0,(e^2-1)/(e^2+1))$ then the above necessary conditions hold. To establish sufficiency, consider $z=e^{i \theta}$, then for $c=(e^2-1)/(e^2+1)$, we have
    $$q_c(\theta):=\dfrac{(2(1+e^4+(e^4-1)\cos\theta))^{\frac{1}{4}}\left(\cos t+i\sin t\right)}{\sqrt{1+e^2}},$$ 
    where $t=\dfrac{1}{2}\arctan\dfrac{(e^2-1)\sin\theta}{1+e^2+(e^2-1)\cos\theta}$. Notably, owing to the symmetry of $q_c(\theta)$ about the real axis, our consideration of $\theta$ is restricted to the range $[0,\pi]$. For $c=(e^2-1)/(e^2+1)$, it can be easily observed  that the squared distance between the boundary curves $q_c(\theta)$ and $\mathfrak{B}(e^{i\theta})$ is always non-zero for $\theta\in(0,\pi)$. Consequently, the curves do not intersect, affirming that $\Lambda_c$ remains confined within $\mathfrak{B}(\mathbb{D})$. This concludes the proof.
    \end{proof}


\section{Differential Subordination Results}
In this section, we need the following lemma from \cite{miller}:
\begin{lemma}\label{MM}
	\emph{\cite[Theorem 3.4h,p.132]{miller}} Let $g(z)$ be univalent in $\mathbb{D}$,  $\Phi$ and $\Theta$ be  analytic in a domain $\Omega$ containing $g(\mathbb{D})$ such that $\Phi(w)\neq0,$ when $w\in g(\mathbb{D})$. Now letting $Q(z)=zg'(z)\cdot\Phi(g(z))$, $h(z)=\Theta(g(z))+Q(z)$ and either  $h$ is convex  or $G$ is starlike. In addition, if \begin{equation*}
	\RE\frac{zh'(z)}{Q(z)}=\RE\bigg(\frac{\Theta'(g(z))}{\Phi(g(z))}+\frac{zQ'(z)}{Q(z)}\bigg)>0
	\end{equation*}  and let $p$ be analytic  in $\mathbb{D},$ with $p(0)=g(0),$ $p(\mathbb{D})\subset \Omega$ and 
	\begin{equation*}
	\Theta(p(z))+zp'(z)\cdot\Phi( p(z))\prec \Theta(g(z))+zg'(z)\cdot\Phi(g(z))=:h(z)
	\end{equation*}
	then $p\prec g$, and $g$ is the best dominant.
\end{lemma}
Now we have the following lemma which is needed to prove our results:
\begin{lemma}\label{subradi}
	Let $R_{0}=e\sqrt{2/(1+e^2)}\simeq1.32725$. Then
	\begin{equation*}
	\bigg|\log\bigg(\frac{w^2}{2-w^2}\bigg)\bigg|\geq2\quad\text{if and only if} \quad |w|\geq R_{0}.
	\end{equation*}
\end{lemma}
\begin{proof}
	Let $w=R e^{i \theta}$ where $R>0$ and $\theta\in(0,2\pi].$ Then
		\begin{gather*}
	\bigg|\log\bigg(\frac{w^2}{2-w^2}\bigg)\bigg|\geq2
	\intertext{if and only if}
	g(R,\theta):=\bigg(\log\frac{R^2}{\sqrt{4+R^4-4R^2\cos(2\theta)}}\bigg)^2+\bigg(\arctan\bigg(\frac{2\sin(2\theta)}{2\cos(2\theta)-R^2}\bigg)\bigg)^2\geq4.
	\end{gather*}
	$g(R,\theta)$ attains its minimum at $\theta=0$. Thus the above inequality holds if and only if
	\begin{gather*}
	\bigg(\log\bigg(\frac{R^2}{2-R^2}\bigg)\bigg)^2\geq4
	\intertext{which is true if and only if}
	R\geq\sqrt\frac{2e^2}{1+e^2}\simeq1.32725.
	\end{gather*}
	Thus the result holds.
\end{proof}

\begin{theorem}\label{thm1}
Let $\alpha\in[0,1]$, $\gamma\in(0,1]$, $k\in\{-1,0\}$ and $\beta\in\mathbb{C}\backslash\{0\}$ with $\RE \beta>0$. Suppose $p\in\mathcal{H}$ satisfies
 \begin{equation*}
(1-\alpha)p(z)+\alpha p^2(z)+\beta \frac{zp'(z)}{p^k(z)}\prec \sqrt{1+\tanh{z}},
\end{equation*}
 then $$p(z)\prec\left(\dfrac{1+Az}{1+Bz}\right)^\gamma\qquad -1<B<A\leq1,$$ whenever
\begin{equation}\label{2}
    |\beta|\geq\dfrac{1}{\gamma (A-B)}\left(R_0+\alpha\left(\dfrac{1+A}{1+B}\right)^{2\gamma}+(1-\alpha)\left(\dfrac{1+A}{1+B}\right)^{\gamma}\right)\dfrac{(1+A)^{\gamma(k-1)+1}}{(1+B)^{\gamma(k-1)-1}},
\end{equation} where  $R_{0}=e\sqrt{2/(1+e^2)}\simeq1.32725$.
\end{theorem}
\begin{proof}
Let $q(z)=((1+Az)/(1+Bz))^\gamma$. Taking $\Theta(w)=(1-\alpha)w+\alpha w^2$ and $\Phi(w)=\beta/w^k$ in Lemma \ref{MM}, the function  $Q:\mathbb{D}\to \mathbb{C}$ becomes 
$$Q(z)=\dfrac{\beta zq'(z)}{q^k(z)}=\dfrac{\beta\gamma(A-B)z(1+Bz)^{\gamma(k-1)-1}}{(1+Az)^{\gamma(k-1)+1}}.$$
Since, \begin{gather*}
    \dfrac{zQ'(z)}{Q(z)}=1+(\gamma(k-1)-1)\dfrac{Bz}{1+Bz}-(\gamma(k-1)+1)\dfrac{Az}{1+Az}
    \intertext{and after a computation we obtain,}
    \RE\dfrac{zQ'(z)}{Q(z)}\geq\dfrac{1-AB+\gamma(A-B)(1-k)}{(1-A)(1-B)}\geq0,
    \end{gather*} therefore $Q(z)$ is starlike for $k=-1,0,1$. Further, for the function $h$ defined by 
$$h(z):=\Theta(q(z))+Q(z)=(1-\alpha)q(z)+\alpha q^2(z)+Q(z),$$ we have 
$$\dfrac{zh'(z)}{Q(z)}=\dfrac{(1-\alpha)q^k(z)}{\beta}+\dfrac{2\alpha}{\beta}q^{k+1}(z)+\dfrac{zQ'(z)}{Q(z)}.$$ 
Since $\RE q^{k}(z)>0$ for $k=\{-1,0,1\}$, thus $\RE zh'(z)/Q(z)>0$  for $k=\{-1,0\}$. Hence, by Lemma \ref{MM}, we have $p(z)\prec q(z)$, whenever $$(1-\alpha)p(z)+\alpha p^2(z)+\beta \frac{zp'(z)}{p^k(z)}\prec(1-\alpha)q(z)+\alpha q^2(z)+\beta \frac{zq'(z)}{q^k(z)}.$$ Thus, it is enough to show that  now $$\sqrt{1+\tanh{z}}\prec(1-\alpha)q(z)+\alpha q^2(z)+\beta \frac{zq'(z)}{q^k(z)}=h(z).$$
Let $w=\mathfrak{B}(z)=\sqrt{1+\tanh(z)}$. Then $\mathfrak{B}^{-1}(w)=1/2\log(w^2/(2-w^2))$. Therefore the subordination $\mathfrak{B}(z)\prec h(z)$ is equivalent to $z\prec\mathfrak{B}^{-1}(h(z))$. Thus we only need to show $|\mathfrak{B}^{-1}(h(e^{i\theta}))|\geq1$ where $0\leq\theta\leq2\pi$ which is true if and only if
\begin{gather*}
	\left|\log\left(\frac{\left((1-\alpha)\left(\tfrac{1+Ae^{i\theta}}{1+Be^{i\theta}}\right)^\gamma+\alpha\left(\tfrac{1+Ae^{i\theta}}{1+Be^{i\theta}}\right)^{2\gamma}+\tfrac{\beta\gamma(A-B) e^{i\theta}(1+Be^{i\theta})^{\gamma(k-1)-1}}{(1+Ae^{i\theta})^{\gamma(k-1)+1}}\right)^2}{2-\left((1-\alpha)\left(\tfrac{1+Ae^{i\theta}}{1+Be^{i\theta}}\right)^\gamma+\alpha\left(\tfrac{1+Ae^{i\theta}}{1+Be^{i\theta}}\right)^{2\gamma}+\tfrac{\beta\gamma(A-B) e^{i\theta}(1+Be^{i\theta})^{\gamma(k-1)-1}}{(1+Ae^{i\theta})^{\gamma(k-1)+1}}\right)^2}\right)\right|\geq2.
	\intertext{By Lemma \ref{subradi}, it follows that the above inequality holds whenever}
	\bigg|(1-\alpha)\left(\tfrac{1+Ae^{i\theta}}{1+Be^{i\theta}}\right)^\gamma+\alpha\left(\tfrac{1+Ae^{i\theta}}{1+Be^{i\theta}}\right)^{2\gamma}+\tfrac{\beta\gamma(A-B) e^{i\theta}(1+Be^{i\theta})^{\gamma(k-1)-1}}{(1+Ae^{i\theta})^{\gamma(k-1)+1}}\bigg|\geq R_{0},
	\end{gather*}
which is true whenever conditions in (\ref{2}) holds.
\end{proof}

Note that the theorem stated above, Theorem \ref{thm1}, is applicable for $k=1$ whenever $\gamma\in(0,1/2]$. Specifically, we derive the following extended result from Theorem \ref{thm1} when $A=1$, $B=0$, and $\gamma=1/2$.

\begin{theorem}\label{subthm}
	Let $\alpha\in[0,1]$, $k\in\{-1,0,1\}$ and $\beta\in\mathbb{C}\backslash\{0\}$ with $\RE \beta>0$. Suppose $p\in\mathcal{H}$  satisfies
 \begin{equation*}
(1-\alpha)p(z)+\alpha p^2(z)+\beta \frac{zp'(z)}{p^k(z)}\prec \sqrt{1+\tanh{z}},
\end{equation*}
 then $p(z)\prec\sqrt{1+z}$, whenever
\begin{equation*}
    |\beta|\geq(R_0+2\alpha+(1-\alpha)\sqrt{2})2^{(k+3)/2},
\end{equation*} where $R_{0}=e\sqrt{2/(1+e^2)}\simeq1.32725$.
\end{theorem}
Since the proof is much akin to the Theorem \ref{thm1}, it is omitted here.

\begin{theorem}
	Let $\delta\in\{0,1\}$, $k\in\{-1,0,1\}$, $\gamma\in(0,1]$ and $\beta\in\mathbb{C}\backslash\{0\}$ with $\RE \beta>0$. Suppose $p\in\mathcal{H}$ satisfies
	\begin{equation*}
	(p(z))^\delta+\beta \frac{zp'(z)}{(p(z))^k}\prec \sqrt{1+\tanh(z)},
	\end{equation*}
	 then $$p(z)\prec\left(\dfrac{1+Az}{1+Bz}\right)^\gamma\qquad -1< B<A\leq1,$$ whenever
	\begin{equation}\label{janowski-2}
	|\beta|\geq \dfrac{1}{\gamma (A-B)}\left(R_0+\left(\dfrac{1+A}{1+B}\right)^{\delta\gamma}\right)\dfrac{(1+A)^{\gamma(k-1)+1}}{(1+B)^{\gamma(k-1)-1}},
	\end{equation}
 where $R_{0}=e\sqrt{2/(1+e^2)}\simeq1.32725$.
\end{theorem}

\begin{proof}
	Let $q(z)=((1+Az)/(1+Bz))^\gamma$. Taking $\Theta(w)=w^\delta$ and $\Phi(w)=\beta/w^k$ in Lemma \ref{MM}, the function  $Q:\mathbb{D}\to \mathbb{C}$ becomes 
$$Q(z)=\dfrac{\beta zq'(z)}{q^k(z)}=\dfrac{\beta\gamma(A-B)z(1+Bz)^{\gamma(k-1)-1}}{(1+Az)^{\gamma(k-1)+1}}$$
	is starlike  for $k\in\{-1,0,1\}$. Further, for the function $h$ defined by $h(z):=q^\delta(z)+Q(z)$, we have
$$\dfrac{zh'(z)}{Q(z)}=\dfrac{\delta}{\beta}q^{\delta+k-1}(z)+\dfrac{zQ'(z)}{Q(z)}.$$ 
Since $\RE q^{k}(z)>0$ for $k=\{-1,0,1\}$, thus for $\delta\in\{0,1\}$, we have $\RE zh'(z)/Q(z)>0$ whenever $k=\{-1,0,1\}$. Hence, by Lemma \ref{MM}, $$p^\delta(z)+\beta \frac{zp'(z)}{p^k(z)}\prec q^\delta(z)+\beta \frac{zq'(z)}{q^k(z)}=:h(z)$$ implies $p(z)\prec q(z)$. Now to prove our desired result, it is enough to show that  $\sqrt{1+\tanh{z}}\prec h(z).$
Let $w=\mathfrak{B}(z)=\sqrt{1+\tanh(z)}$. Then $\mathfrak{B}^{-1}(w)=1/2\log(w^2/(2-w^2))$. Therefore the subordination $\mathfrak{B}(z)\prec h(z)$ is equivalent to $z\prec\mathfrak{B}^{-1}(h(z))$. Thus we only need to show $|\mathfrak{B}^{-1}(h(e^{i\theta}))|\geq1$ where $\theta\in[0,2\pi)$ which is true if and only if
	\begin{gather*}
	\left|\log\left(\frac{\left(\left(\tfrac{1+Ae^{i\theta}}{1+Be^{i\theta}}\right)^{\delta\gamma}+\tfrac{\beta\gamma(A-B) e^{i\theta}(1+Be^{i\theta})^{\gamma(k-1)-1}}{(1+Ae^{i\theta})^{\gamma(k-1)+1}}\right)^2}{2-\left(\left(\tfrac{1+Ae^{i\theta}}{1+Be^{i\theta}}\right)^{\delta\gamma}+\tfrac{\beta\gamma(A-B) e^{i\theta}(1+Be^{i\theta})^{\gamma(k-1)-1}}{(1+Ae^{i\theta})^{\gamma(k-1)+1}}\right)^2}\right)\right|\geq2.
	\intertext{By Lemma \ref{subradi}, it follows that the above inequality holds whenever}
	\bigg|\left(\tfrac{1+Ae^{i\theta}}{1+Be^{i\theta}}\right)^{\delta\gamma}+\tfrac{\beta\gamma(A-B) e^{i\theta}(1+Be^{i\theta})^{\gamma(k-1)-1}}{(1+Ae^{i\theta})^{\gamma(k-1)+1}}\bigg|\geq R_{0},
	\end{gather*}
which is true whenever conditions in \ref{janowski-2} holds.
	\end{proof}

\section{Radius Results}

We begin with the following result:
	
\begin{theorem}\label{jan}
The sharp $\mathcal{S}^*_{\mathfrak{B}}-$radius for the class  $\mathcal{S}^*[A,B]$ is given by\\
$(i)$ $R_{\mathcal{S}^*_{\mathfrak{B}}}(\mathcal{S}^*[A,B])=\min\left\{1;\frac{\sqrt{2}e-\sqrt{1+e^2}}{A\sqrt{1+e^2}-\sqrt{2}eB}\right\}=:R_1$, whenever $-1\leq B<0<A\leq 1$.\\
$(ii)$ $R_{\mathcal{S}^*_{\mathfrak{B}}}(\mathcal{S}^*[A,B])=\begin{cases}
    R_1& R_1\leq R_0\\
    R_2& R_1>R_0,
\end{cases}$ whenever $0\leq B<A\leq1$,\\
where $$R_0:=\sqrt{\frac{\sqrt{2(1+e^2)}-e-1}{B(\sqrt{2(1+e^2)}A-Be-B)}}\;\;\text{and}\;\;R_2:=\min\left\{1,\frac{\sqrt{1+e^2}-\sqrt{2}}{A\sqrt{1+e^2}-\sqrt{2}B}\right\}.$$
\end{theorem}
\begin{proof}
Let $f\in\mathcal{S}^*[A,B].$ Then by  \cite{janowski}, we have \begin{equation}\label{jand}
    \left|\dfrac{zf'(z)}{f(z)}-c\right|\leq\dfrac{(A-B)r}{1-B^2r^{2}},
\end{equation} where $$c:=\frac{1-ABr^{2}}{1-B^2r^{2}}.$$
Case $(i)$: If $-1\leq B<0<A\leq 1$, then $c\geq1.$
 We observe that $f\in\mathcal{S}^*_{\mathfrak{B}},$ if  the disk (\ref{jand}) is contained in the disk (\ref{bdisk}). Thus, by Lemma \ref{circle}, it suffices to show that 
 $$\dfrac{(A-B)r}{1-B^2r^{2}}\leq e\sqrt{\dfrac{2}{1+e^2}}-\dfrac{1-ABr^{2}}{1-B^2r^{2}},$$
 which yields $$r\leq\frac{\sqrt{2}e-\sqrt{1+e^2}}{A\sqrt{1+e^2}-\sqrt{2}eB}.$$
Case $(ii)$: If $0\leq B<A\leq 1$, then $c\leq1.$ Further, if $R_1\leq R_0,$ then $c\geq(e+1)/\sqrt{2(1+e^2)}\simeq0.9077$ if and only 
if $r\leq R_0$. Therefore, in particular for $r\leq R_1,$ we deduce that $c\geq(e+1)/\sqrt{2(1+e^2)}$.  Hence by Lemma \ref{circle},
we see that $f\in\mathcal{S}^*_{\mathfrak{B}},$ if $$\dfrac{(A-B)r}{1-B^2r^{2}}\leq e\sqrt{\dfrac{2}{1+e^2}}-\dfrac{1-ABr^{2}}{1-B^2r^{2}},$$ i.e, if  $r\leq R_1.$ Now if $R_1> R_0,$ then $c\leq(e+1)/\sqrt{2(1+e^2)}$ if and only 
if $r\geq R_0$. Therefore, in particular for $r\geq R_1,$ we have $c\leq(e+1)/\sqrt{2(1+e^2)}$. Now again by Lemma \ref{circle}, we see that $f\in\mathcal{S}^*_{\mathfrak{B}}$ if
$$\dfrac{(A-B)r}{1-B^2r^{2}}\leq \dfrac{1-ABr^{2}}{1-B^2r^{2}}- \sqrt{\dfrac{2}{1+e^2}},$$ i.e, if  $r\leq R_2.$ The result follows with sharpness due to the function $f_{A,B}(z)$, given by \begin{equation*}
    f_{A,B}(z)=\begin{cases}
        z(1+B z)^{\frac{A-B}{B}};&B\neq0,\\
        z\exp(Az);&B=0,
    \end{cases}
\end{equation*}
which completes the proof.
\end{proof}

\begin{corollary}
The sharp $\mathcal{S}^*_{\mathfrak{B}}-$radius for $\mathcal{S}^*(\alpha)$ is $\frac{\sqrt{2}e-\sqrt{1+e^2}}{(1-2\alpha)\sqrt{1+e^2}+\sqrt{2}e}$, $\alpha\in[0,1)$. In particular, $R_{\mathcal{S}^*_{\mathfrak{B}}}(\mathcal{S}^*)=\frac{\sqrt{2}e-\sqrt{1+e^2}}{\sqrt{2}e+\sqrt{1+e^2}}\simeq0.1406$. The bound is sharp for $k_\alpha(z)=z/(1-z)^{2(1-\alpha)}$. 
\end{corollary}

\begin{remark}
    Note that for $A,B\in[-1,1]$ where $A\neq B$, $(1+Az)/(1+Bz)$ maps unit disk onto 
    $$H(\mathbb{D}):=\left\{w\in\mathbb{C}:\left|w-\dfrac{1-AB}{1-B^2}\right|<\dfrac{|A-B|}{1-B^2}\right\}.$$
   Clearly, $(1+Az)/(1+Bz)$ and  $(1-Az)/(1-Bz)$ represent the same region $H(\mathbb{D})$, hence the radius result for all the possibilities of $A$ and $B$ are incorporated in the Theorem \ref{jan}.  
\end{remark}

 \begin{theorem}
The sharp $\mathcal{S}^*_{\mathfrak{B}}-$radius for the classes   $\mathcal{S}^*_e$, $\mathcal{S}^*_{SG}$,  $\mathcal{S}^*_{\leftmoon}$ and $\mathcal{S}^*_{\varrho}$ are given by\\
$(i)$  $R_{\mathcal{S}^*_{\mathfrak{B}}}(\mathcal{S}^*_{e})=r_{e}=1+\log(\sqrt{2/(1+e^2)})\simeq0.28311$.\\
$(ii)$  $R_{\mathcal{S}^*_{\mathfrak{B}}}(\mathcal{S}^*_{SG})=r_{SG}=-\log(e^{-1}\sqrt{2(1+e^2)}-1)\simeq0.679492.$\\
$(iii)$ $R_{\mathcal{S}^*_{\mathfrak{B}}}(\mathcal{S}^*_{\leftmoon})=r_{\leftmoon}=(e^2-1)/(2e\sqrt{2(e^2+1)})\simeq0.2869.$\\
$(iv)$ $R_{\mathcal{S}^*_{\mathfrak{B}}}(\mathcal{S}^*_{\varrho})=r_{\varrho}\simeq0.253877,$ where $r_{\varrho}$ is the smallest positive root of the equation $$\sqrt{1+e^2}(1+r e^r)=\sqrt{2}e.$$
\end{theorem}

\begin{proof}
$(i)$  Let $f\in\mathcal{S}^*_{e},$ then $zf'(z)/f(z)\prec e^z$. Since for  $|z| \leq r$, we have $$\left|\dfrac{zf'(z)}{f(z)}-1\right|\leq e^r-1,$$ thus by lemma \ref{circle},  it becomes evident that for $r\leq r_{e},$ the disk described above will be contained within $\mathfrak{B}(\mathbb{D})$.  The sharpness of the result can be verified by the  function, given by
\begin{equation*}
	f_0(z)=z\exp\bigg(\int_0^z\dfrac{e^t-1}{t}dt\bigg).
\end{equation*}
    $(ii)$  If $f\in\mathcal{S}^*_{SG},$ then $zf'(z)/f(z)\prec2/(1+e^{-z})=:\psi(z)$.   We need to find the value of $r$ such that $zf'(z)/f(z)\subset\mathfrak{B}(\mathbb{D})$ for $z\in\mathbb{D}_{r}.$ It suffices to determine the value of $r$ that satisfies $\psi(\mathbb{D}_{r})\subset\mathfrak{B}(\mathbb{D})$. For it to hold, it is necessary that
$$\dfrac{2}{1+e^{-r}}=\psi(r)\leq\mathfrak{B}(1)=e\sqrt{\dfrac{2}{1+e^2}},$$
which gives $$r\leq-\log(e^{-1}\sqrt{2(1+e^2)}-1)=:r_{SG}.$$ In fact, the image of the disk $\mathbb{D}_{r_{SG}}$ under the function $\psi(z)$ lies inside the domain $\mathfrak{B}(\mathbb{D}),$ can be seen in the figure \ref{rsg}.\\
\begin{figure}[h]
        \centering
        \includegraphics[height=4cm,width=6cm]{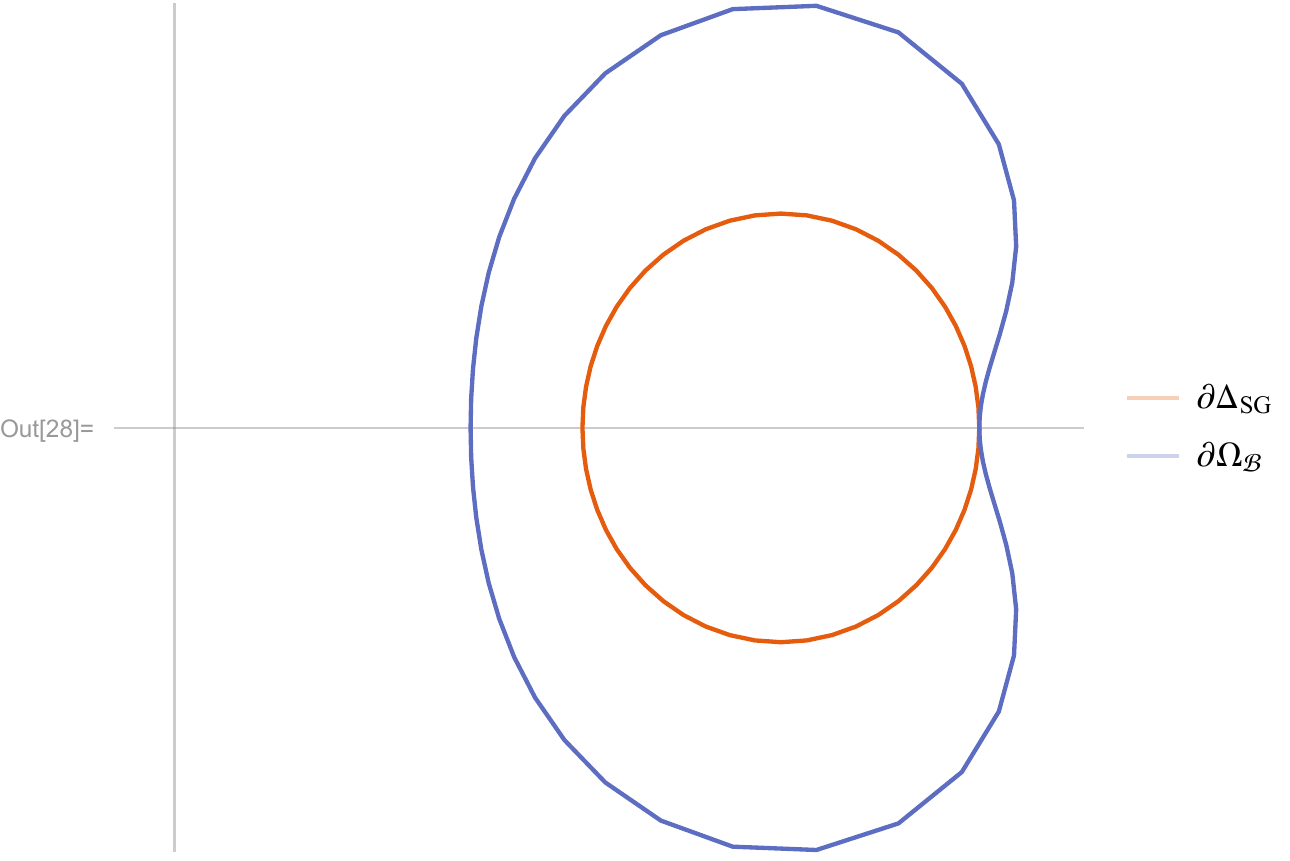}
        \caption{Sigmoid domain $\psi(\mathbb{D}_{r_{SG}})$ is sharply contained in $\mathfrak{B}(\mathbb{D}).$}
        \label{rsg}
    \end{figure}
     $(iii)$ Let $f\in\mathcal{S}^*_{\leftmoon},$ we have $zf'(z)/f(z)\prec z+\sqrt{1+z^2}$. Since for $|z|\leq r$ $$\left|\dfrac{zf'(z)}{f(z)}-1\right|\leq r+\sqrt{1+r^2}-1,$$ thus in view of Lemma \ref{circle} the above disk  lies inside the domain $\mathfrak{B}(\mathbb{D})$ if $r+\sqrt{1+r^2}-1\leq e\sqrt{2/(1+e^2)}-1$. This further implies $r\leq r_{\leftmoon}.$
     The result is sharp for the function
\begin{equation*}
  f_{\leftmoon}(z)=z\exp{\left(\int_0^z\dfrac{t+\sqrt{1+t^2}-1}{t}dt\right)}=z+z^2+\dfrac{3z^3}{4}+\dfrac{5z^4}{12}+\dfrac{z^5}{6}+\cdots.
\end{equation*}
$(iv)$ Let $f\in\mathcal{S}^*_{\varrho},$ then $zf'(z)/f(z)\prec1+ze^z$. For $|z|\leq r,$ we have
$$\left|\dfrac{zf'(z)}{f(z)}-1\right|\leq re^r.$$
Thus by lemma \ref{circle}, it is clear that for the above disk to lie in $\mathfrak{B}(\mathbb{D})$, we need $re^r\leq e\sqrt{2/(1+e^2)}-1,$ which implies that $r\leq r_{\varrho}.$ The sharpness of the result can be verified by the function $f_{\varrho}(z)=ze^{e^z-1}.$
\end{proof}

In 2019, Khatter et al. \cite{khatter} extended the previously defined classes $\mathcal{S}^*_{L}:=\mathcal{S}^*(\sqrt{1+z})$ and $\mathcal{S}^*_e:=\mathcal{S}^*(e^z)$ to encompass broader families of functions. Specifically, they introduced the classes $\mathcal{S}^*_{L}(\alpha):=\mathcal{S}^*(\alpha+(1-\alpha)\sqrt{1+z})$ and $\mathcal{S}^*_{\alpha,e}:=\mathcal{S}^*(\alpha+(1-\alpha)e^z)$, where $\alpha\in[0,1).$

\begin{theorem}
The sharp $\mathcal{S}^*_{\mathfrak{B}}-$radius for the classes   $\mathcal{S}^*_{L}(\alpha)$ and $\mathcal{S}^*_{e}(\alpha),$ are given by\\
$(i)$ $R_{\mathcal{S}^*_{\mathfrak{B}}}(\mathcal{S}^*_L(\alpha))=r_{L}(\alpha):=1-((\sqrt{2/(1+e^2)}-\alpha)/(1-\alpha))^2,$ where $\alpha\in[0,1-e/\sqrt{2(1+e^2)}).$\\
$(ii)$ $R_{\mathcal{S}^*_{\mathfrak{B}}}(\mathcal{S}^*_{e}(\alpha))=r_{e}(\alpha):=\log\left((e\sqrt{2/(1+e^2)}-\alpha)/(1-\alpha)\right),$ where $\alpha\in[0,1).$
\end{theorem}
\begin{proof}
$(i)$ If $f\in\mathcal{S}^*_L(\alpha),$ then $zf'(z)/f(z)\prec\alpha+(1-\alpha)\sqrt{1+z}=:\psi_{L,\alpha}(z)$. We need to find the value of $r$ such that $zf'(z)/f(z)\subset\mathfrak{B}(\mathbb{D})$ for $z\in\mathbb{D}_{r}.$ It suffices to determine the value of $r$ that satisfies $\psi_{L,\alpha}(\mathbb{D}_{r})\subset\mathfrak{B}(\mathbb{D})$. For it to hold, it is necessary that
$$\sqrt{\dfrac{2}{1+e^2}}=\mathfrak{B}(-1)\leq\psi_{L,\alpha}(-r)=\alpha+(1-\alpha)\sqrt{1-r},$$
which gives $$r\leq1-\left(\dfrac{1}{1-\alpha}\left(\sqrt{\dfrac{2}{1+e^2}}-\alpha\right)\right)^2=:r_{L}(\alpha).$$ In fact, the image of the disk $\mathbb{D}_{r_{L}(\alpha)}$ under the function $\psi_{L,\alpha}(z)$ lies inside the domain $\mathfrak{B}(\mathbb{D}),$ can be seen in the figure \ref{rl}, for some particular cases of $\alpha$.
\begin{figure}[ht]
    \centering
    \subfloat[\centering]{\includegraphics[height=4.5cm,width=6cm]{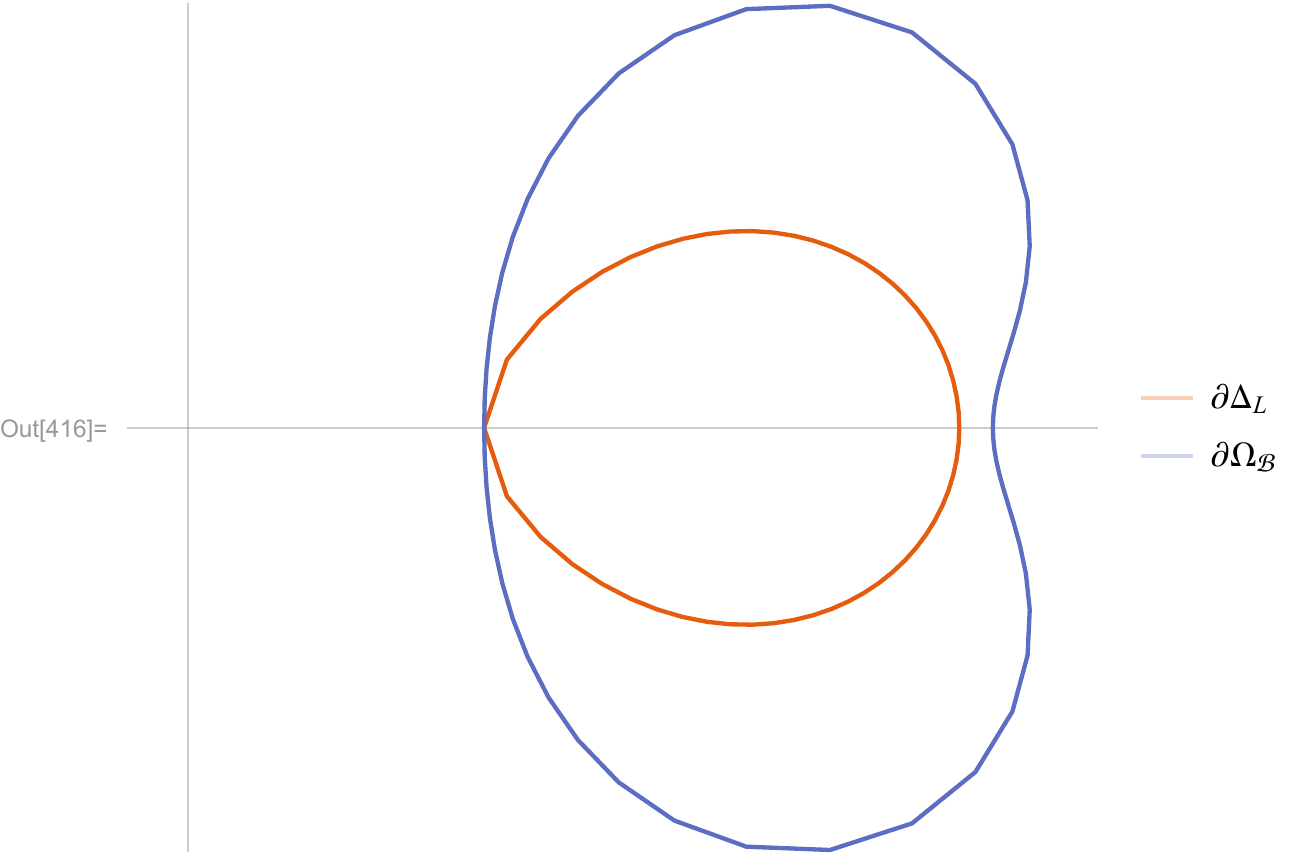}}
\qquad
   \subfloat[ \centering]{\includegraphics[height=4.5cm,width=6cm]{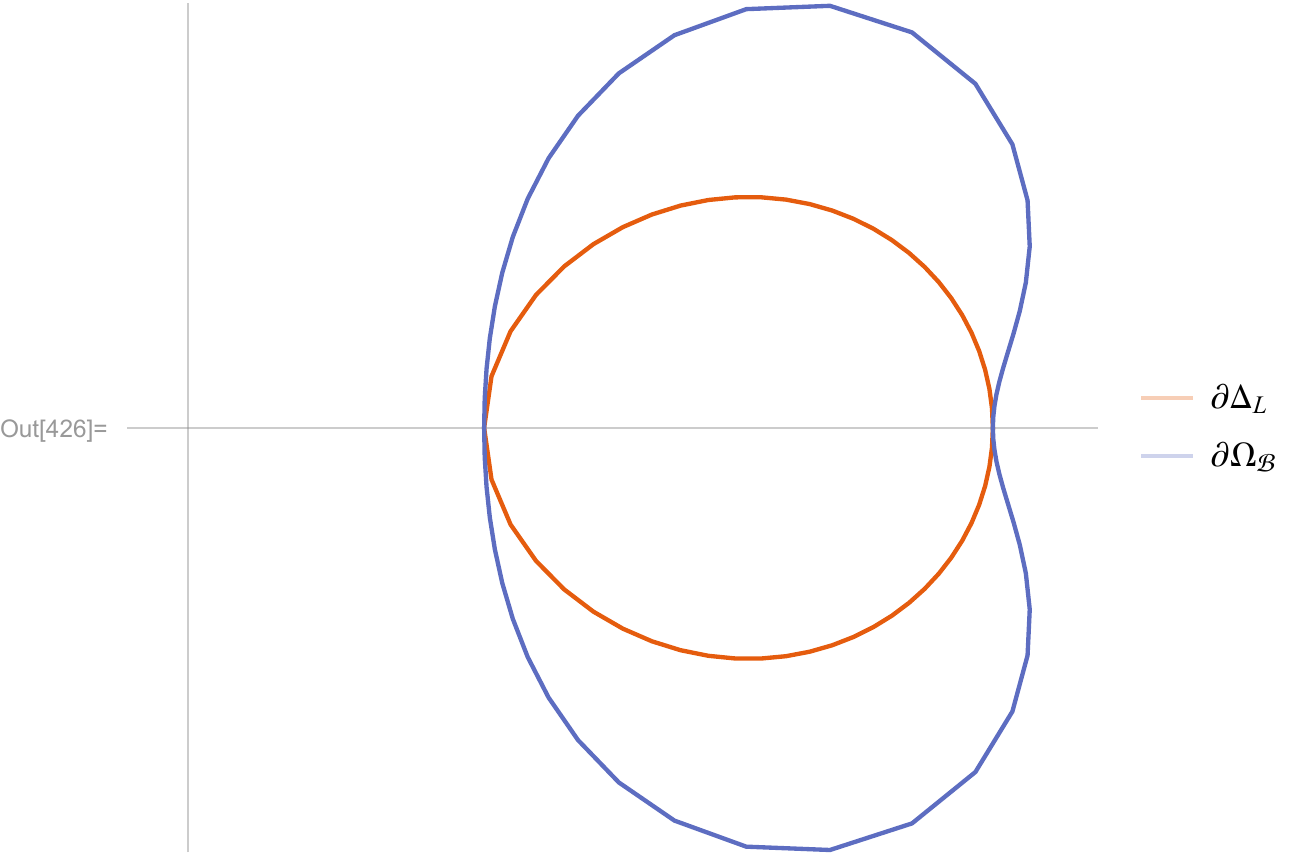}}
    \caption{(A) Sharpness for $\alpha=0.3$, $\psi_{L,0.3}(\partial\mathbb{D}_{r_{L}(0.3)})\subseteq\mathfrak{B}(\mathbb{D})$ and (B) Sharpness for $\alpha=0$, $\psi_{L,0}(\partial\mathbb{D}_{r_{L}(0)})\subseteq\mathfrak{B}(\mathbb{D})$.}
    \label{rl}
\end{figure}   Thus the result holds. \\

$(ii)$ If $f\in\mathcal{S}^*_{\alpha,e},$ then $zf'(z)/f(z)\prec\alpha+(1-\alpha)e^z=:\psi_{e,\alpha}(z)$. We need to find the value of $r$ such that $zf'(z)/f(z)\subset\mathfrak{B}(\mathbb{D})$ for $z\in\mathbb{D}_{r}.$ It suffices to determine the value of $r$ that satisfies $\psi_{e,\alpha}(\mathbb{D}_{r})\subset\mathfrak{B}(\mathbb{D})$. For it to hold, it is necessary that
$$\alpha+(1-\alpha)e^r=\psi_{L,\alpha}(r)\leq\mathfrak{B}(1)=e\sqrt{\dfrac{2}{1+e^2}},$$
which gives $$r\leq\log\left(\dfrac{1}{1-\alpha}\left(e\sqrt{\dfrac{2}{1+e^2}}-\alpha\right)\right)=:r_{e}(\alpha).$$ In fact, the image of the disk $\mathbb{D}_{r_{e}(\alpha)}$ under the function $\psi_{e,\alpha}(z)$ lies inside the domain $\mathfrak{B}(\mathbb{D}),$ can be seen in the figure \ref{ex1}, for some particular cases of $\alpha$.
\begin{figure}[ht]
    \centering
    \subfloat[\centering]{\includegraphics[height=4.5cm,width=6cm]{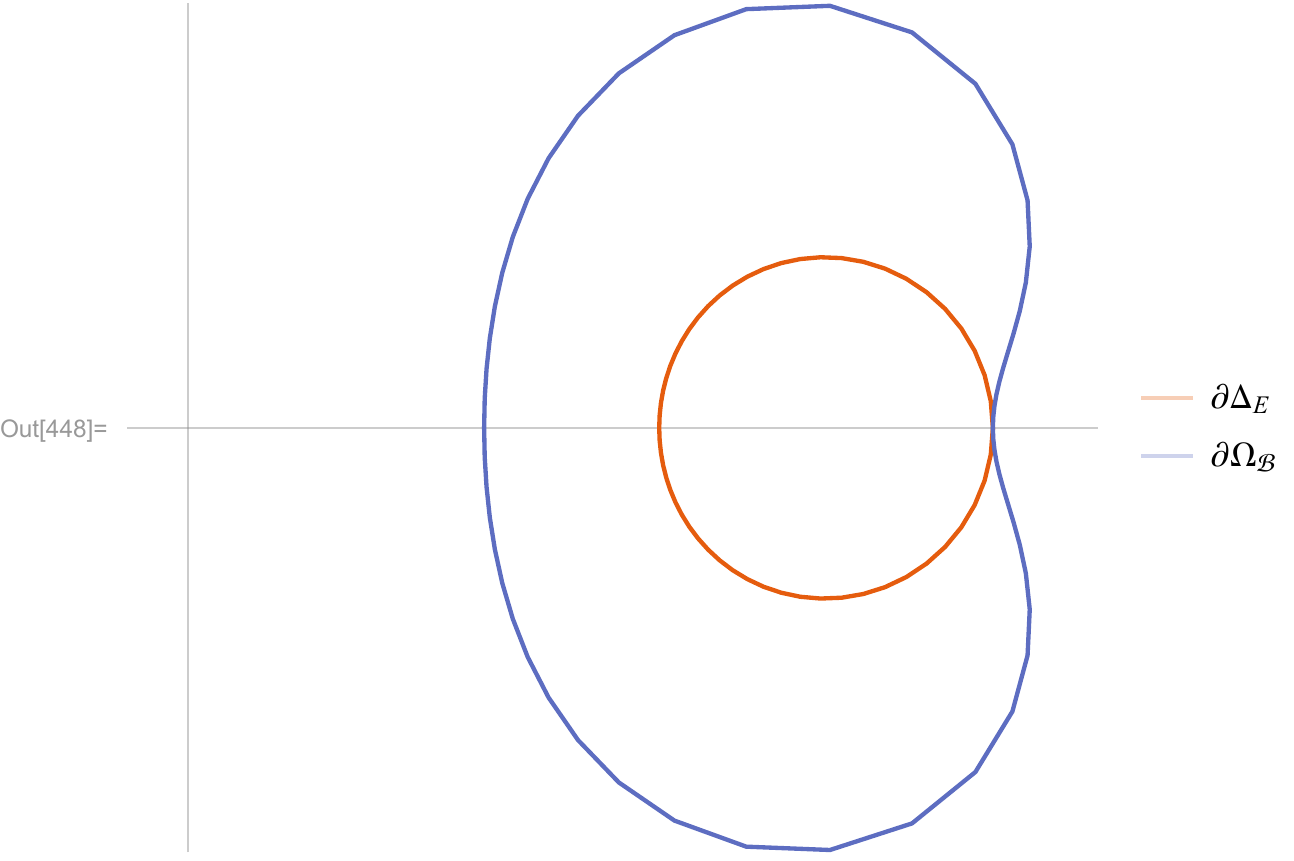}}
\qquad
   \subfloat[ \centering]{\includegraphics[height=4.5cm,width=6cm]{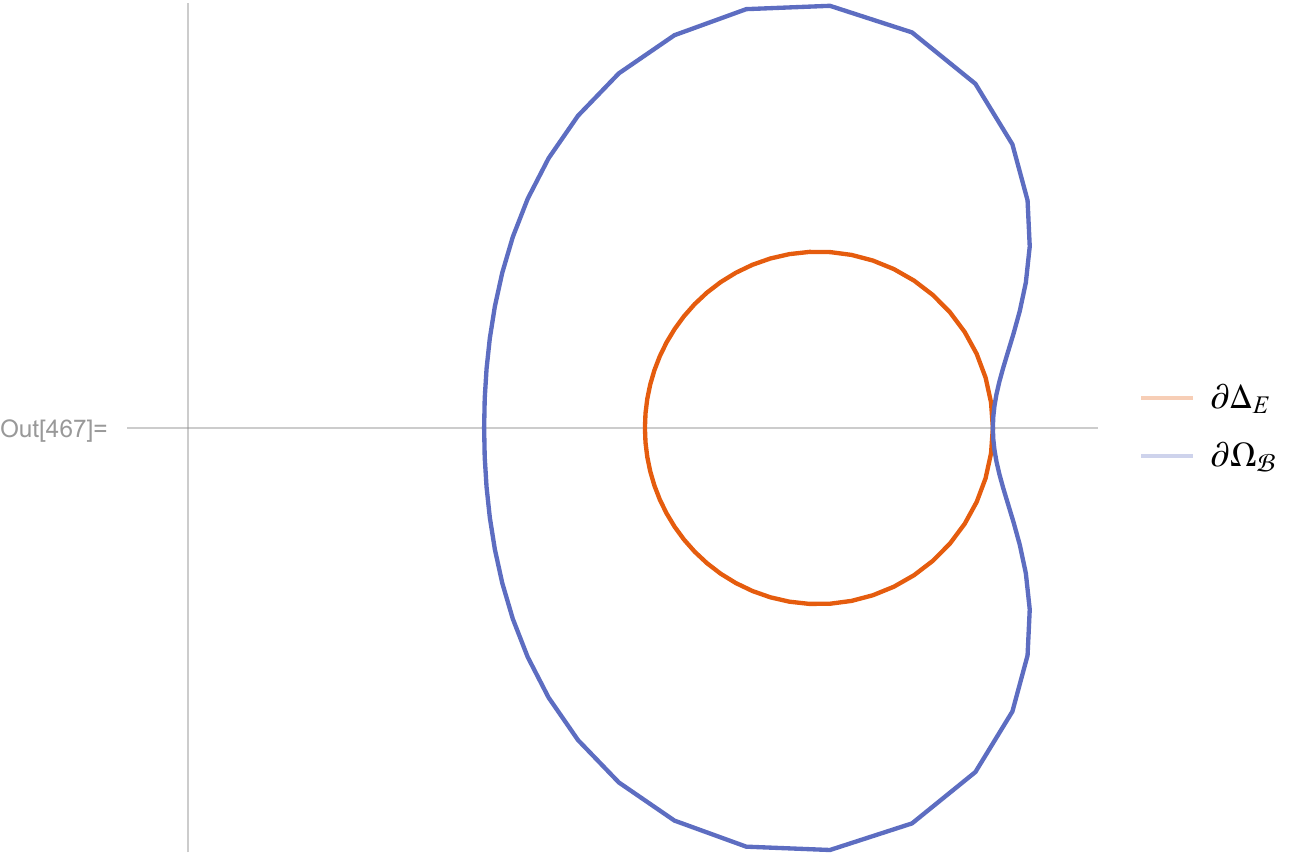}}
    \caption{(A) Sharpness for $\alpha=0.3$, $\psi_{e,0.3}(\partial\mathbb{D}_{r_{e}(0.3)})\subseteq\mathfrak{B}(\mathbb{D})$ and (B) Sharpness for $\alpha=0$, $\psi_{e,0}(\partial\mathbb{D}_{r_{e}(0)})\subseteq\mathfrak{B}(\mathbb{D})$.}
    \label{ex1}
\end{figure}   Thus the result holds. 
\end{proof}

Before we proceed to our next result, we need to recall the following classes:\\
For $0\leq\alpha<1,$ $\mathcal{BS}^*(\alpha):=\{f\in\mathcal{A}:zf'(z)/f(z)\prec1+z/(1-\alpha z^2)\}$ and $\mathcal{S}^*_{cs}(\alpha):=\{f\in\mathcal{A}:zf'(z)/f(z)\prec1+z/((1-z)(1+\alpha z))\}$ defined in \cite{kargar} and \cite{masih} by Kargar et al. and Masih et al. respectively. 
Apart from that, Masih and Kanas \cite{limacon} extensively studied the class $\mathcal{ST}^*_{L}(s)$, which is associated with the Lima\c{c}on of Pascal, defined as $\{f\in\mathcal{A}:zf'(z)/f(z)\prec\mathbb{L}_s(z)\}$, where $\mathbb{L}_s(z)=(1+sz)^2$ and $s\in[-1,1]\backslash\{0\}$.
\begin{theorem}
The sharp $\mathcal{S}^*_{\mathfrak{B}}-$radius for the classes $\mathcal{BS}^*(\alpha),$  $\mathcal{S}^*_{cs}(\alpha)$ and $\mathcal{ST}_L(s)$ are given by\\
$(i)$ $R_{\mathcal{S}^*_{\mathfrak{B}}}(\mathcal{BS}^*(\alpha))=r_{BS}(\alpha)=1-\sqrt{1+4\alpha(1-e\sqrt{2/(1+e^2)})^2}/(2\alpha(1-e\sqrt{2/(1+e^2)})),$ where $\alpha\in[0,1).$\\
$(ii)$ $R_{\mathcal{S}^*_{\mathfrak{B}}}(\mathcal{S}^*_{cs}(\alpha))=r_{cs}(\alpha)$, where $r_{cs}(\alpha)$ is the smallest positive root of
$$\dfrac{(1+e^2)(1+\alpha r-\alpha r^2)}{(1-r)(1+\alpha r)}-e\sqrt{2(1+e^2)}=0,$$ where $\alpha\in[0,1).$\\
$(iii)$ $R_{\mathcal{S}^*_{\mathfrak{B}}}(\mathcal{ST}_L(s))=\min\{1,r^*\}$, where $r^*$ is the smallest positive root of $$(1+e^2)(1+sr)^2-e\sqrt{2(1+e^2)}=0,$$ where $|s|\leq1\backslash\{0\}.$
\end{theorem}
\begin{proof}
$(i)$ Let $f\in\mathcal{BS}^*(\alpha).$ Therefore $zf'(z)/f(z)\prec1+z/(1-\alpha z^2)$. Since
$$\left|\dfrac{zf'(z)}{f(z)}-1\right|\leq\dfrac{r}{1-\alpha r^2}, \qquad |z|\leq r$$ 
thus in view of Lemma \ref{circle}, the above disk  lies inside the domain $\mathfrak{B}(\mathbb{D})$ whenever $r/(1-\alpha r^2)\leq e\sqrt{2/(1+e^2)}-1$. This further implies $r\leq r_{BS}(\alpha).$ \\
$(ii)$ Let $f\in\mathcal{S}^*_{cs}(\alpha),$ then $zf'(z)/f(z)\prec1+ z/((1-z)(1+\alpha z))$. Since 
$$\left|\dfrac{zf'(z)}{f(z)}-1\right|\leq\dfrac{r}{(1-r)(1+\alpha r)},\qquad |z|\leq r $$
thus by lemma \ref{circle}, it is clear that for the above disk to lie in $\mathfrak{B}(\mathbb{D})$, we need $r/((1-r)(1+\alpha))\leq e\sqrt{2/(1+e^2)}-1,$ which implies that $r\leq r_{cs}(\alpha).$ \\
$(iii)$ Let $f\in\mathcal{ST}_L(s).$ Therefore $zf'(z)/f(z)\prec(1+sz)^2$. Since 
$$\left|\dfrac{zf'(z)}{f(z)}-1\right|\leq(1+sr)^2-1,\qquad |z|\leq r$$ thus in view of Lemma \ref{circle} the above disk  lies inside the domain $\mathfrak{B}(\mathbb{D})$ if $(1+sr)^2\leq e\sqrt{2/(1+e^2)}$. This further implies $r\leq\min\{1,r^*\}.$ 
\end{proof}

\begin{theorem}
		If $f\in\mathcal{S}^*_{\mathfrak{B}}$. Then
		\begin{itemize}
			\item [$(i)$] $f\in\mathcal{C}(\alpha)$ in $|z|<r_\alpha$, where $r_\alpha$ is the smallest positive root of the equation
			\begin{equation}\label{c-alpha}
			1-\tanh{r}=\left(\alpha+\dfrac{r(1+\tan{r})}{2(1-r^2)}\right)^2,\quad \alpha\in (0,1].
			\end{equation}
			\item [$(ii)$] $f\in\mathcal{S^{*}}(\alpha)$ in $|z|<(\log{({(2-\alpha^2)}/{\alpha^2})})/2$, where $\alpha\in(\sqrt{{2}/{(1+e^{2})}},1)$.
		\end{itemize}
	\end{theorem}
	
	\begin{proof}
		Let $f\in\mathcal{S}^*_{\mathfrak{B}}$, then there exist a  Schwarz function $w(z)$ such that
		\begin{equation}\label{7}
		\frac{zf'(z)}{f(z)}=\sqrt{1+\tanh{w(z)}}.
		\end{equation}
	$(i)$ By differentiating \eqref{7}, we obtain
		\begin{equation*}
		1+\frac{zf''(z)}{f'(z)}=\sqrt{1+\tanh{w(z)}}+\frac{z\sech^2{w(z)}w'(z)}{2(1+\tanh{w(z)})}.
		\end{equation*}
		Taking the real parts in the above equation and considering $w(z)=R e^{i \theta}$ with  $R\leq r=|z|$, we get
		\begin{align*}
		\RE\left(1+\frac{zf''(z)}{f'(z)}\right)
		&=\RE\sqrt{1+\tanh{w(z)}}+\RE\bigg(\frac{z\sech^2{w(z)}w'(z)}{2(1+\tanh{w(z)})}\bigg)\\
		&\geq\RE\sqrt{1+\tanh{w(z)}}-\frac{|z||\sech^2{w(z)}||w'(z)|}{2|1+\tanh{w(z)}|}\\
		&\geq\sqrt{1+\tanh{R}}-\frac{r(1+|\tanh{w(z)}|)(1-|w(z)|^2)}{2(1-|z|^2)}\\
		&\geq\sqrt{1-\tanh{r}}-\frac{r(1+\tan{R})(1-R^2)}{2(1-r^2)}\\
		&\geq\sqrt{1-\tanh{r}}-\frac{r(1+\tan{r})}{2(1-r^2)}
		=:h(r).
		\end{align*}
  By computing the derivative of $h(r)$, it is evident that $h(r)$ is a decreasing function in the interval $[0,1)$ with $h(0)=1$. Consequently, for $f\in\mathcal{C}(\alpha)$, it is necessary that $h(r)\geq\alpha$, which follows directly from \eqref{c-alpha}. \\
		$(ii)$ 
		For $f\in\mathcal{S}^*_{\mathfrak{B}}$, observe that
		\begin{align*}
		\RE\frac{zf'(z)}{f(z)}=\RE\sqrt{1+\tanh{w(z)}} \geq \sqrt{1-\tanh{|z|}}=\sqrt{\frac{2}{1+e^{2|z|}}}.
		\end{align*}
		Given that $\tanh|z|$ is an increasing function, for $|z|<\frac{1}{2}\log{\frac{2-\alpha^2}{\alpha^2}}=\arctanh(1-\alpha^2)$, $\alpha\in(\sqrt{{2}/{(1+e^{2})}},1)$, we have $\tanh{|z|}\leq1-\alpha^2$, leading to
\begin{equation*}
\RE\frac{zf'(z)}{f(z)}\geq\sqrt{1-(1-\alpha^2)}=\alpha.
\end{equation*}
The result is sharp for the function $f_{0}$ defined in \eqref{5}.
	\end{proof}

\end{document}